\documentclass[graybox]{svmult}

\usepackage{type1cm}        
\usepackage{makeidx}         
\usepackage{graphicx}        
\usepackage{multicol}        
\usepackage[bottom]{footmisc}

\usepackage{newtxtext}       %
\usepackage{newtxmath}       

\usepackage{amsmath}
\usepackage{amstext}
\usepackage{amscd}
\usepackage{arydshln}
\usepackage{color}
\usepackage{graphicx}
\usepackage{longtable}
\usepackage{calc}
\usepackage{microtype}

\makeindex             
                       
\newcommand{\la}{\lambda}
\newcommand{\diag}[1]{\mbox{${\rm diag}(#1)$}}

\newcommand{\hide}[1]{}
\definecolor{myred}{cmyk}{0.000000,1.000000,1.000000,0.1}
\definecolor{myblue}{cmyk}{1.000000,0.750000,0.000000,0.1}

\begin{document}

\title*{Linear System Matrices of Rational Transfer Functions}
\author{Froil\'an Dopico, Mar\'ia del Carmen Quintana and Paul Van Dooren}
\institute{Froil\'an Dopico \at Universidad Carlos III de Madrid, \email{dopico@math.uc3m.es}
\and Mar\'ia del Carmen Quintana \at Universidad Carlos III de Madrid, \email{maquinta@math.uc3m.es}
\and Paul Van Dooren \at Universit\'e catholique de Louvain, \email{paul.vandooren@uclouvain.be}}
%
%
\maketitle

\abstract{ In this paper we derive new sufficient conditions for a linear system matrix 
	$$S(\lambda):=\begin{bmatrix} T(\lambda) \;&\; -U(\lambda) \\ V(\lambda) \;&\; W(\lambda) 
	\end{bmatrix},$$ where $T(\lambda)$ is assumed regular, to be strongly irreducible. In particular, we introduce the notion of strong minimality, and the corresponding conditions are shown to be sufficient for a polynomial system matrix to be strongly minimal. A strongly irreducible or minimal system matrix has the same structural elements as the rational matrix $$R(\lambda):= W(\lambda) + V(\lambda)T(\lambda)^{-1}U(\lambda),$$ which is also known as the transfer function connected to the system matrix $S(\lambda)$. The pole structure, zero structure and null space structure of $R(\lambda)$ can be then computed with the staircase algorithm and the $QZ$ algorithm applied to pencils derived from $S(\lambda)$. 
	We also show how to derive a strongly minimal system matrix from an arbitrary linear system matrix by applying to it a reduction procedure, that only uses unitary equivalence transformations. This implies that numerical errors performed during the reduction procedure remain bounded. Since we use unitary transformations in both the reduction procedure and the computation of the eigenstructure, this guarantees that we computed the exact eigenstructure of a perturbed linear system matrix, but where the perturbation is of the order of the machine precision. }

\section{Introduction}
\label{sec:Introduction}

Already in the seventies, Rosenbrock \cite{Ros70} introduced the concept of a polynomial system matrix
\begin{equation}  \label{Ros}
	S(\la):=\left[\begin{array}{ccc} T(\la) \;&\; -U(\la) \\ V(\la) \;&\; W(\la)  \end{array}\right],
\end{equation}
where $T(\la)$ is assumed to be regular. He showed that the finite pole and zero structure 
of its transfer function matrix $R(\la)= W(\la) + V(\la)T(\la)^{-1}U(\la)$ can be retrieved from the polynomial matrices $T(\la)$ and $S(\la)$, respectively, provided it is {\em irreducible} or {\em minimal}, meaning that the matrices
\begin{equation}  \label{minimal} 
	\left[\begin{array}{ccc} T(\la) \;&\; -U(\la)  \end{array}\right], \quad \left[\begin{array}{ccc} T(\la) \\ V(\la)   \end{array}\right],
\end{equation}
have, respectively, full row and column rank for all finite $\la$. This was already well known for state-space models of a proper transfer function $R_p(\la)$,
where the system matrix takes the special form
$$S_p(\la):=\left[\begin{array}{ccc} \la I-A \;&\; -B \\ C \;&\; D \end{array}\right]$$ 
where $(A,B)$ is controllable and $(A,C)$ is observable, meaning that $S_p(\la)$ is minimal. That is, $\left[\begin{array}{ccc} \la I-A \;&\; -B \end{array}\right]$ and $ \left[\begin{array}{ccc} \la I-A \\ C \end{array}\right]$ both satisfy the conditions in \eqref{minimal}, respectively. The poles of such a proper transfer function are all finite and are the eigenvalues of $A$, while the finite zeros are the finite generalized eigenvalues of the pencil  $S_p(\la)$. The main advantage of using state-space models is that there are algorithms to 
compute the eigenstructure using unitary transformations only. There are also algorithms available to derive 
a minimal state-space model from a non-minimal one, and these algorithms are also based on unitary transformations only 
\cite{Van81}. 

\medskip

When allowing generalized state space models, then all transfer functions can be realized by a system matrix of the type
\begin{equation} \label{gss}
	S_g(\la):=\left[\begin{array}{ccc} \la E-A \;&\; -B \\ C \;&\; D \end{array}\right],
\end{equation} 
since the matrix $E$ is allowed to be singular. Moreover, when the pencils
\begin{equation}  \label{gssminimal}
	\left[\begin{array}{ccc} \la E-A \;&\; -B \end{array}\right], \left[\begin{array}{ccc} \la E-A \\ C \end{array}\right],
\end{equation}
have, respectively, full row rank and column rank for all finite $\la$, then we retrieve the irreducibility or minimality conditions 
of Rosenbrock in \eqref{minimal}, which imply that the finite poles of $R(\la):=D+C(\la E -A)^{-1}B$ are the finite eigenvalues of $\la E-A$ and the finite zeros of $R(\la)$ are the finite zeros of $S_g(\la)$. It was shown in \cite{VVK79} that when imposing also the conditions that the pencil in \eqref{gss} is {\em strongly} irreducible, meaning that the matrices in \eqref{gssminimal} have full row rank for all finite	and infinite $\lambda$, then also the infinite pole and zero structure of $R(\la)$ can be retrieved from the infinite structure of $\la E-A$ and $S_g(\la)$, respectively, and that the left and right minimal indices of $R(\la)$ and $S_g(\la)$ are also the same. Moreover, a reduction procedure to derive a strongly irreducible generalized state-space model from a reducible one
was also given in \cite{Van81}, and it is also based on unitary transformations only.

\medskip

In \cite{Ver80} these results were then extended
to arbitrary polynomial models, but the procedure required irreducibility tests that were more involved. In this paper we will show that these conditions can again be simplified (and also made more uniform) when the system matrix is linear, i.e.,
\begin{equation}  \label{pencil}
	S(\la):=\left[\begin{array}{ccc} A(\la) \;&\; -B(\la) \\ C(\la) \;&\; D(\la) 
	\end{array}\right]:=\left[\begin{array}{ccc} \la A_1 -A_0 \;&\; B_0- \la B_1 \\ \la C_1-C_0 \;&\; \la D_1-D_0  \end{array}\right].
\end{equation} 
We will define the notion of strongly minimal polynomial system matrix, and we will prove that the strong minimality conditions imply the strong irreducibility conditions in \cite{Ver80}. We remark that, although the notions of irreducible or minimal polynomial system matrix refer to the same conditions in \eqref{minimal}, the conditions for a polynomial system matrix to be strongly irreducible or strongly minimal are different in general.
We will also show that when the strong minimality conditions are not satisfied, we can reduce the system matrix to one where they are satisfied, and this without modifying the transfer function. Such a procedure was already derived in \cite{VanD83}, but only for linear system matrices that were already minimal at finite points. In this paper we thus extend this to arbitrary linear system matrices.

In the next Section we briefly recall the background material for this paper and introduce the basic notation. In Section \ref{Minimality} we also recall the definition of strongly irreducible polynomial system matrix in \cite{Ver80}, and we introduce the notion of strong minimality. In addition, we establish the relation between them. We then give, in Section \ref{Reducing}, an algorithm to construct a strongly minimal linear system matrix from an arbitrary one, and we discuss the computational aspects in Section \ref{Computation}. Finally, we end with some numerical experiments in Section \ref{experiments} and some concluding remarks in Section \ref{Conclusion}. 

\section{Background} \label{Background}

We will restrict ourselves here to polynomial and rational matrices with coefficients in the field of complex numbers $\mathbb{C}$. The set of $m\times n$ polynomial matrices, denoted by $\mathbb{C}[\la]^{m\times n}$ and the set of  $m\times n$ rational matrices, denoted by  $\mathbb{C}(\la)^{m\times n}$,
can both be viewed as matrices over the field of rational functions with complex coefficients, denoted by $\mathbb{C}(\la)$. 

Every rational matrix can have poles and zeros and has a right and a left null space (these can be trivial, i.e., equal to $\{0\}$). Via the local Smith-McMillan form, one can associate 
structural indices to the poles and zeros, and via the notion of minimal polynomial bases for rational vector spaces, one can associate 
so called right and left minimal indices to the right and left null spaces. We briefly recall here these different types of indices. Since we assumed (for simplicity)
that the coefficients of the rational matrix are in $\mathbb{C}$, the poles and zeros are in the same set. 

\begin{definition} A square rational matrix $M(\la)\in \mathbb{C}(\la)^{m\times m}$ is said to be regular at a point $\la_0\in \mathbb{C}$ if the matrix 
	$M(\la_0)$ is bounded (i.e., $M(\la_0)\in \mathbb{C}^{m\times m}$) and is invertible. This is equivalent to that both rational matrices $M(\la)$ and $M(\la)^{-1}$ having a convergent Taylor expansion around
	the point $\la=\la_0$. Namely,
	\begin{eqnarray*}
		M(\la)& := & M_{0}+(\la-\la_0)M_{1}+(\la-\la_0)^2M_2+(\la-\la_0)^3M_3+\cdots,\\
		M(\la)^{-1} & := & M_{0}^{-1}+(\la-\la_0)H_{1}+(\la-\la_0)^2H_2+(\la-\la_0)^3H_3+\cdots.
	\end{eqnarray*} If $\la=\infty$, $M(\la)$ is said to be biproper or regular at infinity if the Taylor expansions above are in terms of $1/\la$ instead of the factor $(\la -\la_0).$
\end{definition}

\begin{definition} Let $R(\la)$ be an arbitrary  $m\times n$ rational matrix of normal rank $r$. Then its {\em local Smith-McMillan form} at a point $\lambda_0\in \mathbb{C}$ is the diagonal matrix obtained under rational left and right transformations  $M_{\ell}(\lambda)$ and $M_r(\lambda)$, that are {\em regular} at $\lambda_0$:
	\begin{equation} \label{snf} 
		M_{\ell}(\lambda)R(\lambda)M_r(\lambda) = \left[\begin{array}{cc}
		\diag{(\la-\la_{0})^{d_{1}} ,\ldots, (\la-\la_{0})^{d_{r}}} & 0 \\
		0& 0_{(m-r)\times (n-r)}
		\end{array}\right],
	\end{equation}
	where $d_1\le d_2 \le \ldots \le d_r$. If $\lambda_0=\infty$, the basic factor $(\lambda-\lambda_0)$ is replaced by 
	$\frac{1}{\lambda}$ and the transformation matrices are then biproper. The latter can be viewed as a change of variable $\mu=\frac{1}{\lambda}$  which transform $\lambda_0=\infty$ to $\mu_0=0$.
\end{definition}

\begin{remark}\label{def_polardeg}
	The normal rank of a rational matrix is the size of its largest nonidentically zero minor. The indices  $d_i$ are unique and are called the
	{\em structural indices} of $R(\lambda)$ at $\lambda_0$. In particular, the strictly positive indices correspond to a zero at $\la_0,$ and the strictly negative indices correspond to a pole at $\la_0$. The {\em zero degree} is defined as the sum of all structural indices of all zeros (infinity included), and the {\em polar degree} is the sum of all structural indices (in absolute value) of all poles (infinity included).
\end{remark}

	\begin{example}{Example}
		Let us consider the $2\times 2$ rational matrix
		\begin{equation} \label{Ex1}
			R(\la) = \left[\begin{array}{cc} e_5(\lambda) \;&\; 0 \\ c/\la \;&\; e_1(\lambda) \end{array}\right] 
		\end{equation}
		where $e_5(\lambda)$ is a monic polynomial of degree 5 and $e_1(\lambda)$ is a monic polynomial of degree 1, with $e_5(0)\neq 0$ and $e_1(0)\neq 0$. If $c\neq 0$, the only poles are 0 and infinity, and the corresponding local Smith-McMillan forms for these two points are 
		$$ \lambda_0=0 : \diag{\lambda^{-1},\lambda^1}, \quad \lambda_0=\infty \; (\mu_0=0) : \diag{\mu^{-5},\mu^{-1}},
		$$
		indicating that $\lambda_0=0$ is a zero as well as a pole. The other finite zeros are the six finite roots of $e_5(\lambda)$ and $e_1(\lambda)$.
		The polar degree and the zero degree for this example are thus both equal to 7. When $c=0$, the pole and zero at $\lambda=0$ disappear and the matrix is polynomial instead of rational. The polar and zero degree are then both equal to 6. 
	\end{example}

The above definitions of pole and zero structure of a rational matrix $R(\la)$ are those that are commonly used in linear systems theory (see \cite{Ros70}) and are due to McMillan. They describe the spectral properties of a rational matrix. But when applying them to matrix pencils $S(\la)$ we may wonder if they coincide with definitions of eigenvalues and generalized eigenvalues and their multiplicities, i.e. the Kronecker structure of $S(\la)$ (see \cite{Gan59}). 

\begin{definition} The Kronecker canonical form of an arbitrary $m\times n$ pencil $\lambda B-A$ of normal rank $r$ is a block diagonal form obtained via invertible transformations $S$ and $T$~:
		$$ S(\lambda B-A)T = \diag{L^T_{\eta_1}(\la), \ldots, L^T_{\eta_{m-r}}(\la), \lambda I_{r_f} - A_J, \lambda N - I_{r_\infty}, L_{\epsilon_1}(\lambda), \ldots, L_{\epsilon_{n-r}}(\la) }$$
		where $A_J$ is in Jordan form, $N$ is nilpotent and in Jordan form, and 
		$$L_k(\la):=\left[\begin{array}{cccc} \la & 1 \\  & \ddots & \ddots \\  & & \la & 1 \end{array}\right]  $$ is a $k\times(k+1)$ singular pencil. 
		The finite eigenvalues of $(\la B-A)$ are the $r_f$ eigenvalues of $A_J$ and its $r_\infty$ infinite eigenvalues are the generalized eigenvalues of 
		$\lambda N-I_{r_\infty}$.
	\end{definition}

For this comparison, we only need to look at zeros, since a pencil has only one pole (namely, infinity) and its multiplicity is the rank of the 
coefficient of $\la$. In other words, its polar structure is trivial. But what about the correspondence of the {\em zero structure} of $S(\la)$ (in the McMillan sense) and the eigenvalue structure of $S(\la)$ (in the sense of Kronecker)? It turns out that for finite eigenvalues of $S(\la)$ there is a complete isomorphism with the zero structure of $S(\la)$: every Jordan block of size $k$ at an eigenvalue $\lambda_0$ in the Kronecker canonical form of $S(\la)$ corresponds to an elementary divisor $(\la-\la_0)^k$ in the Smith-McMillan form of $S(\la)$. But for $\la=\infty$, there is a difference. It is well known (see \cite{VVK79}) that a Kronecker block of size $k$ at $\la=\infty$ corresponds to an elementary divisor $(\frac{1}{\la})^{(k-1)}$ in the Smith-McMillan form. For the point at infinity there is thus a shift of 1 in the structural indices. For this reason we want to make a clear distinction between both index sets. Whenever we talk about {\em zeros}, we refer to the McMillan structure, and whenever we talk about {\em eigenvalues}, we refer to the Kronecker structure.

\medskip

It is well known that every rational vector subspace $\mathcal{V}$,
i.e., every subspace $\mathcal{V} \subseteq \mathbb{C}(\la)^n$ over the field $\mathbb{C}(\la)$,
has bases consisting entirely of polynomial vectors. Among them some are minimal in the following sense introduced by Forney \cite{For75}: a {\em minimal basis} of $\mathcal{V}$ is a basis of $\mathcal{V}$ consisting of polynomial vectors whose sum of degrees is minimal among all bases of $\mathcal{V}$  consisting of polynomial vectors. The fundamental property \cite{For75,Kai80} of such bases is that the ordered list of degrees of the polynomial vectors in any minimal basis of $\mathcal{V}$ is always the same. Therefore, these degrees are an intrinsic property of the subspace $\mathcal{V}$ and are called the {\em minimal indices} of $\mathcal{V}$.
This leads to the definition of the minimal bases and indices of a rational matrix. An $m\times n$ rational matrix $R(\lambda)$ of normal rank $r$ smaller than $m$ and/or $n$ has non-trivial left and/or right rational null-spaces, respectively, over the field $\mathbb{C}(\la)$:
\begin{eqnarray*}
	{\cal N}_{\ell}(R)\!\! &:=& \!\left\{y(\la)^T \in\mathbb{C}(\la)^{1 \times m}
	\,:\,y(\la)^T R(\la)\equiv0^T\right\},\\
	{\cal N}_r(R)\!\! &:=& \!\left\{x(\la)\in\mathbb{C}(\la)^{n \times 1}
	\,:\, R(\la)x(\la)\equiv0\right\}.
\end{eqnarray*}
Rational matrices with non-trivial left and/or right null-spaces are said to be {\em singular}. If the rational subspace ${\cal N}_{\ell}(R)$ is non-trivial, it has minimal bases and minimal indices, which are called the {\em left minimal bases and indices} of $R(\la)$. Analogously, the {\em right minimal bases and indices} of $R(\la)$ are those of ${\cal N}_{r}(R)$, whenever this subspace is non-trivial. Notice that an $m\times n$ rational matrix of normal rank
$r$ has  $m-r$ left minimal indices $\{ \eta_1, \ldots, \eta_{m-r}\}$, and $n-r$ right minimal indices $\{ \epsilon_1, \ldots, \epsilon_{n-r}\}$.

The {\em McMillan degree} $\delta(R)$ of a rational matrix $R(\la)$ is the polar degree introduced in Remark \ref{def_polardeg}. 
The following degree sum theorem was proven in \cite{VVK79}, and relates the McMillan degree to the other structural elements
of $R(\la)$: to the the {\em zero degree} $\delta_z(R)$, to the {\em left nullspace degree} $\delta_\ell(R)$, that is the sum of all left minimal indices, and to the {\em right nullspace degree} $\delta_r(R)$, that is the sum of all right minimal indices.
\begin{theorem}\label{th_degsumtheorem}
	Let $R(\la)\in  \mathbb{C}(\la)^{m\times n}$. Then
	$$ \delta(R) := \delta_p(R)= \delta_z(R)+ \delta_\ell(R)+ \delta_r(R).$$
\end{theorem}

\section{Strong irreducibility and minimality} \label{Minimality}

In this section we recall the strong irreducibility conditions in \cite{Ver80} for polynomial system matrices, and we introduce the notion of strong minimality. Then, we study the relation between them for the case of linear system matrices.

\begin{definition}\label{irred}
	A polynomial system matrix $S(\la)$ as in \eqref{Ros} is said to be
	{\em strongly controllable} and {\em strongly observable}, respectively, if the polynomial matrices
	\begin{equation}  \label{condverghese} \left[\begin{array}{ccc} T(\la) \;&\; -U(\la) \;&\; 0 \\ V(\la) \;&\; W(\la) \;&\;-I  \end{array}\right], \quad \mathrm{and} \quad
		\left[\begin{array}{ccc} T(\la) \;&\; -U(\la) \\ V(\la) \;&\; W(\la) \\ 0 \;&\; I \end{array}\right],
	\end{equation}
	have no finite or infinite zeros. If both conditions are satisfied $S(\la)$ is said to be {\em strongly irreducible}.
\end{definition}

Let us now consider the transfer function matrix $R(\la)= W(\la) + V(\la)T(\la)^{-1}U(\la)$ of the polynomial system matrix in \eqref{Ros}. In such a case, we also say that the system quadruple $\{T(\la),U(\la),V(\la),W(\la)\}$  {\em realizes} $R(\la).$ Moreover, we say that the system quadruple is {\em strongly irreducible} if the polynomial system matrix is {\em strongly irreducible}. It was shown in \cite{Ver80} that the pole/zero and null space structure of $R(\la)$ can be retrieved from a strongly irreducible system quadruple $\{T(\la),U(\la),V(\la),W(\la)\}$ as follows.
\begin{theorem}
	If the polynomial system matrix $S(\la)$ in \eqref{Ros} is strongly irreducible, then 
	\begin{enumerate}
		\item the zero structure of $R(\la)$ at finite and infinite $\la$ is the same as the zero structure of $S(\la)$ at finite and infinite $\la$,
		\item the pole structure of $R(\la)$ at finite $\la$ is the same as the zero structure at $\la$ of $T(\la)$,
		\item the pole structure of $R(\la)$ at infinity is the same as the zero structure at infinity of
		$$  \left[\begin{array}{ccc} T(\la) \;&\; -U(\la) \;&\; 0 \\ V(\la) \;&\; W(\la) \;&\; -I \\ 0 \;&\; I \;&\; 0 \end{array}\right],
		$$
		\item the left and right minimal indices of $R(\la)$ and $S(\la)$ are the same.
	\end{enumerate}
\end{theorem}
If one specializes this to the generalized state space model \eqref{gss} one retrieves the results of \cite{VVK79}, which are simpler and only involve 
the pencils $(\la E-A)$, \eqref{gss} and \eqref{gssminimal}. We now show that the above conditions can be simplified when the system matrices
are linear as in \eqref{pencil}. First, we present the definition of strongly minimal polynomial system matrix.

\begin{definition}\label{min}
	Let $d$ be the degree of the polynomial system matrix $S(\la)$ in \eqref{Ros}. $S(\la)$ is said to be {\em strongly E-controllable} and {\em strongly E-observable}, respectively, if the polynomial matrices
	\begin{equation}  \label{condlocal} \left[\begin{array}{cc} T(\la) \;&\; -U(\la)  \end{array}\right], \quad \mathrm{and} \quad
		\left[\begin{array}{c} T(\la)  \\ V(\la)  \end{array}\right],
	\end{equation}
	have no finite or infinite\footnote{The eigenvalues at infinity of a polynomial matrix $P(\la)$ considered as a polynomial matrix of grade $g$, with $g\geq \text{degree}\,P(\la)$, are the eigenvalues at zero of $\text{rev}_g P(\la):=\la^{g}P(1/\la)$ (see \cite{GohbergLancasterRodman09}). } eigenvalues, considered as polynomial matrices of grade $d$. If both conditions are satisfied $S(\la)$ is said to be {\em strongly minimal}.
\end{definition}

The letter E in the definition of strong E-controllability and E-observability refers to the condition of the matrices in \eqref{condlocal} not having eigenvalues, finite or infinite. We prove in Proposition \ref{prop.irredandmin} that the strong irreducibility conditions hold if the strong minimality conditions are satisfied. For this, we need  to recall Lemma 1 of \cite{VVK79}, which we give here in its transposed form. Then, we prove Theorems \ref{infeig} and \ref{infeig2}, and Proposition \ref{prop.irredandmin} as a corollary of them.
\begin{lemma} \label{vvklemma}
	The zero structure at infinity of the pencil $\left[ \begin{array}{c|c}\la K_1-K_0 \;&\; -L_0 \end{array}\right]$ where $K_1$ has full column rank, is isomorphic to the zero structure  at zero of the pencil $\left[ \begin{array}{c|c}K_1- \mu K_0 \;&\; -L_0 \end{array}\right]$. Moreover, if the pencil has full row normal rank, then it has no zeros at infinity, provided  the constant matrix  $\left[ \begin{array}{c|c}K_1 \;&\; -L_0 \end{array}\right]$ has full row rank. 
\end{lemma}
\begin{proof} The first part is proven in \cite{VVK79}. The second part is a direct consequence of the first part, when filling in $\mu =0$.
\end{proof}
\begin{theorem} \label{infeig}
	The pencil 
	\begin{equation} \label{cond1}
		\left[\begin{array}{ccc} \la A_1 -A_0 \;&\; B_0- \la B_1  \;&\; 0 \\ \la C_1-C_0 \;&\; \la D_1-D_0 \;&\; -I \end{array}\right],
	\end{equation}
	where $\la A_1 -A_0$ is regular, has no zeros at infinity if the pencil 
	\begin{equation} \label{cond2} 
		\left[\begin{array}{ccc} \la A_1 -A_0 \;&\; B_0- \la B_1  \end{array}\right]
	\end{equation}
	has no eigenvalues at infinity.
\end{theorem}
\begin{proof} 
	Clearly the pencils in \eqref{cond1} and \eqref{cond2} have full row normal rank since  $\la A_1 -A_0$ is regular. We can thus apply the result of Lemma \ref{vvklemma} as follows.
	If we use an invertible matrix $V$ to ``compress'' the columns of the coefficient of $\la$ in the following pencil
	$$ \left[\begin{array}{cc|c} \la A_1 -A_0 \;&\; B_0- \la B_1  \;&\; 0 \\ \la C_1-C_0 \;&\; \la D_1-D_0 \;&\; -I \end{array}\right]
	\left[\begin{array}{c|c} V  \;&\; 0 \\ \hline 0 \;&\; I  \end{array}\right]= 
	\left[\begin{array}{cc|c} \la K_1 -K_0 \;&\; -L_0  \;&\; 0 \\ \la \widehat K_1-\widehat K_0 \;&\; -\widehat L_0 \;&\; -I \end{array}\right],
	$$
	such that the matrix $\left[\begin{array}{c} K_1\\ \widehat K_1 \end{array}\right]$ has full column rank, 
	then this pencil has no zeros at infinity provided the constant matrix 
	$  \left[\begin{array}{cc|c} K_1 \;&\; -L_0  \;&\; 0 \\ \widehat K_1 \;&\; -\widehat L_0 \;&\; -I \end{array}\right]
	$
	has full row rank. But if $\left[\begin{array}{ccc} \la A_1 -A_0 \;&\; B_0- \la B_1  \end{array}\right]$ has no infinite eigenvalues, it follows that 
	$ \left[\begin{array}{cc} A_1 \;&\; - B_1 \end{array}\right]$ has full row rank. And since $ \left[\begin{array}{cc} A_1 \;&\; - B_1 \end{array}\right]V= \left[\begin{array}{cc} K_1 \;&\; 0 \end{array}\right]$, $K_1$ must have full row rank as well (in fact, it is invertible). It then follows from Lemma \ref{vvklemma} that the pencil in \eqref{cond1} has no zeros at infinity. 
\end{proof}

In the next theorem, we state without proof the transposed version of Theorem \ref{infeig}.

\begin{theorem} \label{infeig2}
	The pencil
	\begin{equation*} \label{cond12}
		\left[\begin{array}{ccc} \la A_1 -A_0 \;&\; B_0- \la B_1 \\ \la C_1-C_0 \;&\; \la D_1-D_0 \\ 0 \;&\; I  \end{array}\right],
	\end{equation*}
	where $\la A_1 -A_0$ is regular, has no zeros at infinity if the pencil
	\begin{equation} \label{cond22} 
		\left[\begin{array}{ccc} \la A_1 -A_0 \\ \la C_1-C_0  \end{array}\right]
	\end{equation}
	has no eigenvalues at infinity.
\end{theorem}

Let us now consider a linear system matrix  \begin{equation}\label{linearpencil}
	L(\la) := \la L_1-L_0 := \left[\begin{array}{ccc} \la A_1 -A_0 \;&\; B_0- \la B_1  \\ \la C_1-C_0 \;&\; \la D_1-D_0 \end{array}\right],
\end{equation} 
with $\la A_1 -A_0$ regular. Notice that if $L(\la)$ is minimal (i.e., satisfies \eqref{minimal}) and, in addition, satisfies the conditions in \eqref{cond2} and \eqref{cond22}, then it is strongly minimal.  By Theorems \ref{infeig} and \ref{infeig2}, we have that these conditions imply strong irreducibility on linear system matrices. We state such result in Proposition \ref{prop.irredandmin}.

\begin{proposition} \label{prop.irredandmin} A linear system matrix as in \eqref{linearpencil} is strongly irreducible if it is strongly minimal.
\end{proposition}

\begin{remark} Notice that conditions \eqref{cond2} and \eqref{cond22} are only sufficient, not necessary. But they are easy to test, 
	and also to obtain after a reduction procedure, as we show in Section \ref{Reducing}.
\end{remark}

Theorems \ref{infeig} and \ref{infeig2} and Proposition \ref{prop.irredandmin} can be extended to polynomial system matrices. However, we do not state these results here since, in this paper, we are focusing on linear system matrices. If we recapitulate the results of this section, we obtain the following theorem.
\begin{theorem} \label{minMcM}
	A linear system pencil $L(\la)$ as in \eqref{linearpencil}, realizing the transfer function $R(\la):=(\la D_1-D_0)+(\la C_1-C_0)( \la A_1-A_0)^{-1}( \la B_1-B_0)$, is strongly irreducible if it is strongly minimal. Moreover, if $L(\la)$ is strongly irreducible then 
	\begin{enumerate}
		\item the zero structure of $R(\la)$ at finite and infinite $\la$ is the same as the zero structure of $L(\la)$ at finite and infinite $\la$,
		\item the left and right minimal indices of $R(\la)$ and $L(\la)$ are the same,
		\item the finite polar structure of $R(\la)$ is the same as the finite zero structure of $\la A_1-A_0$, and
		\item the infinite polar structure of $R(\la)$ is the same as the infinite zero structure of the pencil
		\begin{equation}\label{larger}
			\left[\begin{array}{ccc} \la A_1 -A_0 \;&\; - \la B_1 \;&\; 0  \\ \la C_1  \;&\;  \la D_1  \;&\; -I \\
				0 \;&\; I \;&\; 0 \end{array}\right].
		\end{equation}
	\end{enumerate}
\end{theorem}

\begin{remark}\rm It follows from this theorem and the degree sum theorem in Theorem \ref{th_degsumtheorem} that the rank of $L_1$ equals the 
	McMillan degree of $R(\la)$, and that there can be no linear system matrix for $R(\la)$ with a smaller rank of $L_1$ that satisfies Theorem \ref{minMcM}. 
\end{remark}

It may look strange that there is such a difference in the treatment of finite and infinite poles of $R(\la)$ in Theorem \ref{minMcM}, but it should be pointed out that the matrices
$(B_1, C_1, D_1)$ contribute to the infinite polar structure of $R(\la)$, and not to the finite polar structure. Notice that in \eqref{larger} we have eliminated the 
matrices $B_0,C_0$ and $D_0$ with strict equivalence transformations using the identity matrices as pivots.

\section{Reducing to a strongly minimal linear system matrix} \label{Reducing}

In this section we give an algorithm to reduce an arbitrary linear system matrix to a strongly minimal one. Given a linear system quadruple $\{A(\la),B(\la),C(\la),D(\la)\},$ where $A(\la)\in\mathbb{C}(\la)^{d\times d}$, $B(\la)\in\mathbb{C}(\la)^{d\times n}$, $C(\la)\in\mathbb{C}(\la)^{m\times d}$, $D(\la)\in\mathbb{C}(\la)^{m\times n}$ and $A(\la)$ is assumed to be regular, we describe first how to obtain a strongly E-controllable quadruple  $\allowbreak\{A_c(\la),B_c(\la), C_c(\la),D_c(\la)\}$ of smaller state dimension $(d-r)$. For that, our reduction procedure deflates finite and infinite ``uncontrollable eigenvalues'' by proceeding in three different steps. Then the reduction to a strongly E-observable one is dual and can be obtained by mere transposition of the system matrix and application of the first method for obtaining a strongly E-controllable system.

\textbf{Step 1:} We first show that there exist unitary transformations $U$ and $V$ that yield a decomposition of the type
\begin{equation} \label{step1} \left[\begin{array}{cc} U  \;&\; 0 \\  0 \;&\; I_m  \end{array}\right]
	\left[\begin{array}{cc} A(\la) \;&\; - B(\la) \\  C(\la) \;&\; D(\la)  \end{array}\right]
	\left[\begin{array}{cc} V  \;&\; 0 \\  0 \;&\; I _n \end{array}\right]=
	\left[\begin{array}{ccc}  X(\la) \widehat W_{11} \;&\;0 \;&\; X(\la) W_{13} \\  \widetilde Y(\la) \;&\; \widetilde A(\la) \;&\; - \widetilde B(\la) \\ \widetilde Z(\la) \;&\; \widetilde C(\la) \;&\;  D(\la)  \end{array}\right],
\end{equation}
where $\widehat W_{11}\in\mathbb{C}^{r\times r}$ and $W_{13}\in\mathbb{C}^{r\times n}$ are constant, and $\widehat W_{11}$ is invertible. This will allow us in step 2 to deflate the block $X(\la)$ and construct a lower order model that is strongly E-controllable. In order to prove this, we start from the generalized Schur decomposition for singular pencils (see \cite{Van79b})
\begin{equation}\label{schurdec}
	U 
	\left[\begin{array}{cc} A(\la) \;&\; - B(\la)  \end{array}\right]W^\ast=
	\left[\begin{array}{ccc}  X(\la)\;&\; 0 \;&\; 0 \\  Y(\la) \;&\; \widehat A(\la) \;&\;  -\widehat B(\la)  \end{array}\right],
\end{equation}
where $X(\la) \in \mathbb{C}[\la]^{r\times r}$ is the regular part of $\begin{bmatrix}
 A(\la) \;&\; -B(\la)\end{bmatrix},$ $\widehat A(\la) \in \mathbb{C}[\la]^{(d-r)\times(d-r)}$, and $\begin{bmatrix} \widehat A(\la) \;&\; -\widehat B(\la)\end{bmatrix}$ has no finite or infinite eigenvalues anymore. The decomposition in \eqref{schurdec} can be obtained by using unitary transformations $U$ and $W.$ If we partition $U$ as $\left[\begin{array}{c}U_1 \\U_2 \end{array}\right],$ with $U_1\in \mathbb{C}^{r\times d}$, then 
$$ U_1\left[\begin{array}{c|c} A(\la) \;&\; - B(\la)  \end{array}\right]=
\left[\begin{array}{cc|c}  X(\la)W_{11} \;&\; X(\la)W_{12} \;&\; X(\la)W_{13} \end{array}\right],
$$
where $W_{11}\in\mathbb{C}^{r\times r}$, $W_{12}\in\mathbb{C}^{r\times (d-r)}$ and $W_{13}\in\mathbb{C}^{r\times n}$ are the corresponding submatrices of $W$. Since $A(\la)$ is regular, $X(\la)\begin{bmatrix}W_{11} \;&\; W_{12} \end{bmatrix}$ must be full normal rank, and hence $\begin{bmatrix}W_{11} \;&\; W_{12}  \end{bmatrix}$ must be full row rank as well. 
Therefore, there must exist a unitary matrix $V$ such that $\begin{bmatrix}W_{11} \;&\;  W_{12}\end{bmatrix}V= \begin{bmatrix}\widehat W_{11}  \;&\; 0 \end{bmatrix}$, where $\widehat W_{11}$ is invertible. Hence, we have 
\begin{equation*} \label{intermediate} \left[\begin{array}{cc} U \;&\; 0 \\  0 \;&\;  I_{m} \end{array}\right] \left[\begin{array}{cc}  A(\la) \;&\; -B(\la) \\  C(\la) \;&\; D(\la) \end{array}\right] 
	\left[\begin{array}{cc} V \;&\; 0 \\  0 \;&\; I_{n} \end{array}\right]= 
	\left[\begin{array}{ccc} X(\la) \widehat W_{11} \;&\; 0 \;&\; X(\la) W_{13} \\  \widetilde Y(\la) \;&\; \widetilde A(\la) \;&\; - \widetilde B(\la) \\ \widetilde Z(\la) \;&\; \widetilde C(\la) \;&\;  D(\la)  \end{array}\right],
\end{equation*}
where 
$$ W \left[\begin{array}{cc} V \;&\; 0 \\  0 \;&\; I_{n} \end{array}\right]= \left[\begin{array}{ccc} \widehat W_{11} \;&\; 0 \;&\; W_{13}\\
\widehat W_{21} \;&\; \widehat W_{22} \;&\; W_{23} \\ \widehat W_{31} \;&\; \widehat W_{32} \;&\; W_{33}\end{array}\right].
$$

\textbf{Step 2:} We now define $E := -\widehat W_{11}^{-1}W_{13}$ and perform the following non-unitary transformation on the pencil:
$$ \left[\begin{array}{ccc} X(\la)\widehat W_{11} \;&\; 0 \;&\; X(\la) W_{13} \\ \widetilde Y(\la) \;&\; \widetilde A(\la) \;&\;  - \widetilde B(\la) \\  \widetilde Z(\la) \;&\; \widetilde C(\la) \;&\; D(\la) \end{array}\right] \! \left[\begin{array}{ccc} I_{r} \;&\; 0 \;&\; E \\  0 \;&\; I_{d-r} \;&\;  0 \\  0  \;&\; 0 \;&\; I_{n} \end{array}\right]
\! = \! \left[\begin{array}{ccc} X(\la)\widehat W_{11} \;&\; 0 \;&\; 0 \\ \widetilde Y(\la) \;&\;  \widetilde A(\la) &  \widetilde Y(\la)E- \widetilde B(\la) \\  \widetilde Z(\la) \;&\;  \widetilde C(\la) \;&\; \widetilde Z(\la)E  +D(\la) \end{array}\right].
$$
We have obtained an equivalent system representation in which the $(1,1)$-block, $X(\la)\widehat W_{11},$ can be deflated
since it does not contribute to the transfer function. We then obtain
a smaller linear system pencil:
$$ \left[\begin{array}{cc} \widetilde A(\la) \;&\;  \widetilde Y(\la)E- \widetilde B(\la)  \\ \widetilde C(\la) \;&\; \widetilde Z(\la)E+D(\la)   \end{array}\right],
$$ 
that has the same transfer function. One can also perform this elimination by another unitary transformation $\widetilde W$ constructed to eliminate $W_{13}$:
\begin{equation}\label{orthprod}  \left[\begin{array}{ccc} \widehat W_{11} \;&\; 0 \;&\; W_{13} \end{array}\right] \left[\begin{array}{ccc} \widetilde W_{11} & 0 & \widetilde W_{13} \\
		0 \;&\; I_{d-r} \;&\; 0 \\  \widetilde W_{31} \;&\; 0 \;&\; \widetilde W_{33} \end{array}\right] =\left[\begin{array}{ccc} I_{r} \;&\; 0 \;&\; 0 \end{array}\right],
\end{equation}
implying $\widetilde W_{11}=\widehat W_{11}^*$ , $\widetilde W_{31}=W_{13}^*$, and  $\widetilde W_{13}=-\widehat W_{11}^{-1}W_{13}\widetilde W_{33}$.
	This then yields
$$ \left[\begin{array}{ccc} X(\la)\widehat W_{11} \;&\; 0 \;&\;  X(\la)W_{13} \\ \widetilde Y(\la) \;&\;  \widetilde A(\la) \;&\;  - \widetilde B(\la) \\ \widetilde Z(\la) \;&\; \widetilde C(\la) \;&\; D(\la) \end{array}\right] \left[\begin{array}{ccc}  \widetilde W_{11} \;&\; 0 \;&\; \widetilde W_{13} \\
0 \;&\; I_{d-r} \;&\; 0 \\  \widetilde W_{31} \;&\; 0 \;&\; \widetilde W_{33} \end{array}\right]$$
$$
= \left[\begin{array}{ccc} X(\la) \;&\; 0 \;&\; 0 \\ \widetilde Y(\la)\widetilde W_{11} -\widetilde B(\la)\widetilde W_{31} \;&\; \widetilde A(\la) \;&\;  \widetilde Y(\la)\widetilde W_{13} - \widetilde B(\la)\widetilde W_{33} \\ \widetilde Z(\la)\widetilde W_{11}+D(\la)\widetilde W_{31} \;&\; \widetilde C(\la) \;&\; \widetilde Z(\la)\widetilde W_{13}+D(\la)\widetilde W_{33}  \end{array}\right].
$$
Notice that the new transfer function has now changed, but only by postmultiplication by the constant matrix $\widetilde W_{33}$, which moreover is invertible. This follows from 
$$  \left[\begin{array}{c} E\\  I_{n} \end{array}\right] \widetilde W_{33}=  \left[\begin{array}{c} \widetilde W_{13}\\  \widetilde W_{33} \end{array}\right],
$$
expressing that both matrices span the null-space of the same matrix $\left[\begin{array}{cc}\widehat W_{11} \;&\; W_{13} \end{array}\right]$ and where the right hand side matrix has full rank since it has orthonormal columns. This also implies that 
$$ \left[\begin{array}{cc} \widetilde A(\la) \;&\;  \widetilde Y(\la)E- \widetilde B(\la)  \\  \widetilde C(\la) \;&\; \widetilde Z(\la) E+D(\la)   \end{array}\right]
\left[\begin{array}{cc} I_{d-r} \;&\; 0 \\ 0 \;&\;  \widetilde W_{33}\end{array}\right]=
\left[\begin{array}{cc} \widetilde A(\la) \;&\;   \widetilde Y(\la)\widetilde W_{13}- \widetilde B(\la)\widetilde W_{33}  \\  \widetilde C(\la) \;&\; \widetilde Z(\la) \widetilde W_{13}+D(\la)\widetilde W_{33}   \end{array}\right],$$
which shows that their Schur complements are related by the constant matrix $\widetilde W_{33}$. 
\\

\textbf{Step 3:} Finally, we show that the submatrix 
$$ \left[\begin{array}{cc} \widetilde A(\la) \;&\;   \widetilde Y(\la)E- \widetilde B(\la)  \end{array}\right]
\left[\begin{array}{cc} I_{d-r} \;&\; 0 \\ 0 \;&\;  \widetilde W_{33}\end{array}\right]=
\left[\begin{array}{cc} \widetilde A(\la) \;&\;   \widetilde Y(\la)\widetilde W_{13}- \widetilde B(\la)\widetilde W_{33} \end{array}\right],
$$
has no finite or infinite eigenvalues anymore. For this, we first point out that the following product of unitary matrices has the form given below
	$$  W  
	\left[\begin{array}{ccc}  V \;&\; 0 \\ 0 \;&\; I_{n} \end{array}\right]  \left[\begin{array}{ccc}  \widetilde W_{11} \;&\; 0 \;&\; \widetilde W_{13} \\
	0 \;&\; I_{d-r} \;&\; 0 \\  \widetilde W_{31} \;&\; 0 \;&\; \widetilde W_{33} \end{array}\right] =: 
	\left[\begin{array}{ccc}  I_{r} \;&\; 0 \;&\; 0 \\
	0 \;&\;  \widetilde V_{22} \;&\; \widetilde V_{23} \\  0 \;&\; \widetilde V_{32}  \;&\; \widetilde V_{33} \end{array}\right] =:  \left[\begin{array}{ccc}  I_{r} \;&\; 0  \\ 0 \;&\; \widetilde V \end{array}\right]  
	$$
	because the identity \eqref{orthprod} implies that the first block column equals $\begin{bmatrix}
 I_r \;&\;  0 \;&\;  0 	\end{bmatrix}$. This then implies the equality
$$  \left[\begin{array}{ccc} X(\la) \;&\; 0 \;&\; 0 \\ \widetilde Y(\la)\widetilde W_{11} -\widetilde B(\la)\widetilde W_{31} \;&\;  \widetilde A(\la) \;&\; \widetilde Y(\la)\widetilde W_{13} - \widetilde B(\la)\widetilde W_{33} \end{array}\right]  $$
$$
=\left[\begin{array}{ccc}  X(\la) \;&\; 0 \;&\; 0 \\  Y(\la) \;&\; \widehat A(\la) \;&\;  -\widehat B(\la)  \end{array}\right]
\left[\begin{array}{ccc}  I_{r} \;&\; 0  \\ 0 \;&\; \widetilde V \end{array}\right],
$$
which in turn implies that $\left[\begin{array}{cc} \widetilde A(\la) \;&\;   \widetilde Y(\la)\widetilde W_{13}- \widetilde B(\la)\widetilde W_{33} \end{array}\right]$
has no finite or infinite eigenvalues. We thus have shown that the system matrix
$$ S_c(\la) := \left[\begin{array}{cc} A_c(\la) \;&\; - B_c(\la)  \\  C_c(\la) \;&\; D_c(\la)  \end{array}\right]:=
\left[\begin{array}{cc} \widetilde A(\la) \;&\;  \widetilde Y(\la)\widetilde W_{13} - \widetilde B(\la)\widetilde W_{33} \\  \widetilde C(\la) \;&\; \widetilde Z(\la)\widetilde W_{13}+D(\la)\widetilde W_{33}  \end{array}\right]
$$
is now strongly E-controllable and that its transfer function $R_c(\la)$ equals $R(\la)\widetilde W_{33},$ where $R(\la)$ is the transfer function of the original quadruple and $\widetilde W_{33}$ is invertible. We summarize the result obtained by the three-step procedure above in Theorem \ref{reducec}, where we denote $d-r$ by $d_c$, to indicate that it is the size of $A_c(\la) $ in the new strongly E-controllable system, and $r$ is replaced by $d_{\overline c}$, so that $d=d_{\overline c}+d_c$.

\begin{theorem} \label{reducec}
	Let $\{A(\la),B(\la),C(\la),D(\la)\}$ be a linear system quadruple, with $A(\la)\in \mathbb{C}[\la]^{d\times d}$ regular, realizing the rational matrix 
	$R(\la):=C(\la)A(\la)^{-1}B(\la)+D(\la) \in \mathbb{C}(\lambda)^{m\times n}.$ Then there exist unitary transformations $U,V \in\mathbb{C}^{d\times d}$ and $\widetilde W\in  \mathbb{C}^{(d+n)\times (d+n)}$ such that the following identity holds
	\begin{equation*}
		\left[\begin{array}{cc} U  \;&\; 0 \\  0 \;&\; I_m  \end{array}\right] \!
		\left[\begin{array}{cc} A(\la) \;&\; - B(\la) \\  C(\la) \;&\; D(\la)  \end{array}\right] \!
		\left[\begin{array}{cc} V  \;&\;  0 \\  0 \;&\;  I _n \end{array}\right] \! \widetilde W
		=  \! \left[\begin{array}{ccc}  X_{\overline c}(\la) \;&\; 0 \;&\; 0 \\  Y_{\overline c}(\la) \;&\; A_c(\la) \;&\;  -B_c(\la) \\ Z_{\overline c}(\la) \;&\; C_c(\la) \;&\;  D_c(\la)  \end{array}\right] ,
	\end{equation*}
	where $\widetilde W$ is of the form $\widetilde W:= \left[\begin{array}{ccc} \widetilde W_{11} \;&\; 0 \;&\; \widetilde W_{13} \\ 0 \;&\; I_{d_c} \;&\; 0 \\ \widetilde W_{31} \;&\; 0 \;&\; \widetilde W_{33}  \end{array}\right]\in  \mathbb{C}^{(d_{\overline c}+d_c +n)\times (d_{\overline c} +d_c +n)},$ $d_{\overline c}$ is the number of (finite and infinite) eigenvalues of $\left[\begin{array}{cc}A(\la) \;&\; -B(\la) \end{array}\right],$ and $X_{\overline c}(\la)\in 
	\mathbb{C}[\la]^{d_{\overline c}\times d_{\overline c}}$ is a regular pencil. Moreover,
	\begin{itemize}
		\item[a)] the eigenvalues of $\left[\begin{array}{cc}A(\la) \;&\; -B(\la) \end{array}\right]$ are the eigenvalues of $X_{\overline c}(\la)$,
		\item[b)] $\left[\begin{array}{cc}A_c(\la) \;&\; -B_c(\la) \end{array}\right] \in  \mathbb{C}[\la]^{d_c\times (d_c +n)}$ has no (finite or infinite) eigenvalues, 
		\item[c)] the quadruple  $\{A_c(\la),B_c(\la),C_c(\la),D_c(\la)\}$ is a realization of the transfer function $R_c(\la):=R(\la)\widetilde W_{33}$,
		with $\widetilde W_{33}\in \mathbb{C}^{n\times n}$ invertible, and
		\item[d)] if $\left[\begin{array}{cc}A(\la) \\ C(\la) \end{array}\right]$ has no finite or infinite eigenvalues, then $\left[\begin{array}{cc}A_c(\la) \\ C_c(\la) \end{array}\right]$ also has no finite or infinite eigenvalues. 
	\end{itemize} 
\end{theorem}
\begin{remark}\rm Notice that conditions $b)$ and $d)$ in Theorem \ref{reducec} imply that the system quadruple $\{A_c(\la),B_c(\la),C_c(\la),D_c(\la)\}$ is strongly minimal. 
\end{remark}

\begin{proof} The decomposition and the three properties $a)$, $b)$ and $c)$ were shown in the discussion above. The only part that remains to be proven is property $d)$. This follows from the identity \eqref{step1}, which yields 
	$$  \left[\begin{array}{cc} U \;&\; 0 \\  0 \;&\; I_{m} \end{array}\right] \left[\begin{array}{cc}  A(\la) \\  C(\la) \end{array}\right] V = 
	\left[\begin{array}{cc} X(\la) \widehat W_{11} \;&\; 0 \\  \widetilde Y(\la) \;&\; A_c(\la) \\ \widetilde Z(\la) \;&\; C_c(\la)  \end{array}\right].
	$$
	This clearly implies that if $\left[\begin{array}{cc}A(\la) \\ C(\la) \end{array}\right]$ has full rank for all $\la$ (including infinity),
	then so does  $\left[\begin{array}{cc}A_c(\la) \\ C_c(\la) \end{array}\right]$. 
\end{proof}

We state below a dual theorem that constructs, from an arbitrary linear system quadruple $\{A(\la),B(\la),C(\la),D(\la)\},$ a subsystem $\{A_o(\la),B_o(\la),C_o(\la),D_o(\la)\}$ where $\left[\begin{array}{cc}A_o(\la) \\ C_o(\la) \end{array}\right]$ has no finite or infinite eigenvalues. Its proof is obtained by applying the previous theorem on the transposed system $\{A^T(\la),C^T(\la),B^T(\la),D^T(\la)\}$ and then transposing back the result.
\begin{theorem} \label{reduceo}
	Let $\{A(\la),B(\la),C(\la),D(\la)\}$ be a linear system quadruple, with $A(\la)\in \mathbb{C}[\la]^{d\times d}$ regular, realizing the rational matrix 
	$R(\la):=C(\la)A(\la)^{-1}B(\la)+D(\la) \in \mathbb{C}(\lambda)^{m\times n}.$ Then there exist unitary transformations $U,V \in\mathbb{C}^{d\times d}$ and $\widetilde W\in  \mathbb{C}^{(d+m)\times (d+m)}$ such that the following identity holds
	\begin{equation*}  \widetilde W
		\left[\begin{array}{cc} U  \;&\; 0 \\  0 \;&\;  I_m  \end{array}\right] \!
		\left[\begin{array}{cc} A(\la) \;&\; - B(\la) \\  C(\la) \;&\; D(\la)  \end{array}\right] \!
		\left[\begin{array}{cc} V \;&\; 0 \\  0 \;&\; I _n \end{array}\right] \! 
		=  \! \left[\begin{array}{ccc}  X_{\overline o}(\la) \;&\; Y_{\overline o}(\la) \;&\; Z_{\overline o}(\la) \\  0 \;&\; A_o(\la) \;&\; -B_o(\la) \\  0 \;&\; C_o(\la) \;&\;  D_o(\la)  \end{array}\right],
	\end{equation*} 
	where $\widetilde W$ is of the form $\widetilde W:=\left[\begin{array}{ccc} \widetilde W_{11} \;&\; 0 \;&\; \widetilde W_{13} \\ 0 \;&\; I_{d_o} \;&\; 0 \\ \widetilde W_{31} \;&\; 0 \;&\; \widetilde W_{33}  \end{array}\right]\in  \mathbb{C}^{(d_{\overline o} +d_o+m)\times (d_{\overline o} +d_o+m)},$ $d_{\overline o}$ is the number of (finite and infinite) eigenvalues of $\left[\begin{array}{cc}A(\la) \\ C(\la) \end{array}\right]$, and $X_{\overline o}(\la)\in 
	\mathbb{C}[\la]^{d_{\overline o}\times d_{\overline o}}$ is a regular pencil. Moreover,
	\begin{itemize}
		\item[a)] the eigenvalues of $\left[\begin{array}{cc}A(\la) \\ C(\la) \end{array}\right]$ are the eigenvalues of $X_{\overline o}(\la),$ 
		\item[b)] $\left[\begin{array}{cc}A_o(\la) \\ C_o(\la) \end{array}\right] \in  \mathbb{C}[\la]^{(d_o+m)\times d_o}$ has no (finite or infinite) eigenvalues, 
		\item[c)] the quadruple  $\{A_o(\la),B_o(\la),C_o(\la),D_o(\la)\}$ is a realization of the transfer function $R_o(\la):=\widetilde W_{33}R(\la)$, with $\widetilde W_{33} \in  \mathbb{C}^{m\times m}$ invertible, and
		\item[d)] if $\left[\begin{array}{cc}A(\la) \;&\; -B(\la) \end{array}\right]$ has no finite or infinite eigenvalues then $\left[\begin{array}{cc}A_o(\la) \;&\; -B_o(\la) \end{array}\right]$ also has no finite or infinite eigenvalues. 
	\end{itemize} 
\end{theorem}

In order to extract from the system quadruple $\{A(\la),B(\la),C(\la),D(\la)\}$ a subsystem  $\{A_{co}(\la),B_{co}(\la),C_{co}(\la),D_{co}(\la)\}$ that is both strongly E-controllable and E-observable (and hence also strongly minimal), we only need to apply 
the above two theorems one after the other. The resulting subsystem would then be a realization of the transfer function 
$R_{co}=C_{co}(\la)A_{co}(\la)^{-1}B_{co}(\la)+D_{co}(\la) = W_\ell R(\lambda) W_r \in \mathbb{C}(\lambda)^{m\times n}$.
Since the transfer function was changed only by left and right transformations that are constant and invertible, the left and right nullspace will be transformed by these invertible transformations, but their minimal indices will be unchanged.

\section{Computational aspects} \label{Computation}

In this section we give a more ``algorithmic'' description of the procedure described in Section \ref{Reducing} to reduce a given 
system quadruple $\{A(\la),B(\la),C(\la),D(\la)\}$ to a strongly E-controllable quadruple  $\{A_c(\la),B_c(\la),C_c(\la),D_c(\la)\}$ of smaller size.
We describe the essence of the three steps that were discussed in that section.

\bigskip

\noindent
{\bf Step 1} : Compute the staircase reduction of the submatrix $\left[\begin{array}{cc} A(\la) \;&\; - B(\la)  \end{array}\right]$
$$ U 
\left[\begin{array}{cc} A(\la) \;&\; - B(\la)  \end{array}\right]W^\ast=
\left[\begin{array}{c|cc}  X(\la)\;&\; 0 \;&\;0  \\ \hline Y(\la) \;&\; \widehat A(\la) \;&\;  -\widehat B(\la)  \end{array}\right].
$$
{\bf Step 2} : Compute the unitary matrices $V$ and $\widetilde W$ to compress the first block row of $W$
$$  \left[\begin{array}{ccc} W_{11} \;&\; W_{12} \;&\; W_{13} \end{array}\right]\left[\begin{array}{cc} V \;&\; 0 \\  0 \;&\; I_{n} \end{array}\right]  \left[\begin{array}{ccc} \widetilde W_{11} \;&\; 0 \;&\; \widetilde W_{13} \\
0 \;&\; I_{d-r} \;&\; 0 \\  \widetilde W_{31} \;&\; 0 \;&\; \widetilde W_{33} \end{array}\right] =\left[\begin{array}{ccc} I_{r} \;&\; 0 \;&\; 0 \end{array}\right],
$$
where $V$ does the compression $\begin{bmatrix}W_{11} \;&\; W_{12} \end{bmatrix}V=\begin{bmatrix}\widetilde W_{11}^{*} \;&\; 0\end{bmatrix}$ of the first two blocks and $\widetilde W$ does the further reduction of the first block row to $\begin{bmatrix}I_r \;&\; 0 \;&\; 0\end{bmatrix}$.
\\
{\bf Step 3} : Display the uncontrollable part $X(\la)$ using the transformations $U$, $V$ and $\widetilde W$
$$ \left[\begin{array}{cc} U \;&\; 0 \\  0 \;&\; I_{m} \end{array}\right] \left[\begin{array}{cc}  A(\la) \;&\; -B(\la) \\  C(\la) \;&\; D(\la) \end{array}\right] 
\left[\begin{array}{cc} V \;&\; 0 \\  0 \;&\; I_{n} \end{array}\right]\widetilde W=  \left[\begin{array}{ccc} X_{\overline{c}}(\la) \;&\; 0 \;&\; 0 \\ \times  \;&\;  A_c(\la) \;&\; -B_c(\la) 
\\ \times  \;&\;  C_c(\la) \;&\; D_c(\la)  \end{array}\right],
$$
where we have used the notations introduced in Section \ref{Reducing}, and the resulting $\times$ entries are of no interest because they do not contribute to the transfer function $R_c(\la):=C_c(\la)A_c(\la)^{-1}B_c(\la)+D_c(\la)$.

 The computational complexity of these three steps is cubic in the dimensions of the matrices that are involved, provided that the staircase algorithm is implemented in an efficient manner \cite{BeeV}. But it is also important to point out that the reduction procedure to extract a strongly minimal linear system matrix from an arbitrary one, can be  done with unitary transformations only, and that only one staircase reduction is needed when one knows that 
the pencil $\left[\begin{array}{cc} A(\la) \;&\; - B(\la)  \end{array}\right]$ has normal rank equal to its number of rows. Indeed, this pencil then does not have 
any left null space or left minimal indices and only the regular part has to be separated from the right null space structure. This can be obtained by performing one staircase reduction on the {\em rotated} pencil $\left[\begin{array}{cc} \widetilde A(\mu) \;&\; - \widetilde B(\mu)  \end{array}\right]$, where the coefficient matrices
$$  \left[\begin{array}{cc} \widetilde A_0 \\ \widetilde A_1  \end{array}\right]=  \left[\begin{array}{cc} c I \;&\; s I \\ -sI \;&\; cI \end{array}\right] \left[\begin{array}{cc} A_0 \\ A_1  \end{array}\right], \quad \left[\begin{array}{cc} \widetilde B_0 \\ \widetilde B_1  \end{array}\right]=  \left[\begin{array}{cc} c I \;&\; s I \\ -sI \;&\; cI \end{array}\right] \left[\begin{array}{cc} B_0 \\ B_1  \end{array}\right], \quad c^2+s^2=1
$$ 
correspond to a change of variable $\lambda= (c\mu -s)/(s\mu +c)$. If one now chooses the rotation such that the rotated pencil has no eigenvalues at $\mu=\infty$, then only the finite spectrum has to be separated from the right minimal indices, which can be done with one staircase reduction \cite{Van79b}.

\section{Numerical results}\label{experiments}

	We illustrate the results of this paper with a polynomial example and a rational one.
	
	\begin{example}	{Example 1} We consider the $2\times 2$ polynomial matrix $P(\la)=\diag{e_1(\la),e_5(\la)}$, where $e_5(\lambda)$ is a polynomial of degree 5 with coefficients $[ 9.6367e-01 \;\; -5.4026e-07 \;\;  2.6333e-01 \;\; -1.1101e-04 \;\; -2.9955e-04 \;\; 4.4650e-02]$, ordered by descending powers of $\lambda$, and $e_1(\lambda)$ is a polynomial of degree 1 with coefficients $[-2.1886e-03 \;\;-1.0000e+00]$, that were randomly chosen. Expanding this fifth order polynomial matrix as 
		$$ P(\lambda)=P_0+ P_1\lambda + \cdots + P_5 \lambda^5,$$
		a linear system matrix $S_P(\la)$ of $P(\la)$ is given by the following $10\times 10$ pencil:
		$$ S_P(\lambda) =  \left[\begin{array}{cccc|c} I_2 \;&\; -\lambda I_2 & & & P_1 \\ & I_2 \;&\; -\lambda I_2 & & P_2 \\
		& & I_2 \;&\; -\lambda I_2 & P_3 \\ & & &  I_2 \;&\; P_4 + \lambda P_5 \\ \hline
		-\lambda I_2 & & & & P_0
		\end{array}\right].
		$$
		The six finite Smith zeros of $P(\la)$ are clearly those of the scalar polynomials $e_1(\la)$ and $e_5(\la)$. These are also the finite zeros of $S_P(\la),$ since $S_P(\la)$ is minimal. However, $S_P(\la)$ is not strongly minimal if $P_5$ is singular and, in fact, it has 4 eigenvalues at infinity (in the sense of \cite{GohbergLancasterRodman09}). But in the McMillan sense, $P(\la)$ has 
		{\em no} infinite zeros. The deflation procedure that we derived in this paper precisely gets rid of the {\em extraneous} infinite eigenvalues of $S_P(\la)$. The numerical tests show that the sensitivity of the true McMillan zeros also can benefit from this.
		
		\bigskip
		
		In this example we compare the roots computed by four different methods:
		\begin{enumerate}
			\item computing the roots of the scalar polynomials and appending four $\infty$ roots,
			\item computing the generalized eigenvalues of $S_P(\la)$,
			\item computing the roots of $QS_P(\la)Z$ for random orthogonal matrices $Q$ and $Z$,
			\item computing the roots of the {\em minimal} pencil obtained by our method.
		\end{enumerate}
		
		The first column are the so-called ``correct'' eigenvalues $\la_i$, corresponding to the first method, the next three columns are the corresponding errors $\delta^{(k)}_i := |\la_i-\hat \la^{(k)}_i|$, $k=2,3,4$, of the above three methods\footnote{An error $\delta^{(k)}_i$ is $\text{NaN}$ when it is the indeterminate form $\text{Inf} - \text{Inf}$. However, some of the eigenvalues at $\infty$ are computed as a large but finite number and, then, the corresponding error is $\text{Inf}$.}. The {\em extraneous} eigenvalues that are deflated in our approach are put between brackets. 
		
		\vspace{-0.3cm}
		\begin{table}[ht] 
				\begin{center}	
						\begin{tabular}{c|ccc} 
					$\la_i$ \;&\; $\delta^{(2)}_i$ \;&\; $\delta^{(3)}_i$ \;&\; $\delta^{(4)}_i$ \\ \hline
				-4.5811e-01                         \;&\;   2.7756e-16   \;&\;  4.4409e-16 \;&\;   1.1102e-16 \\
				
				3.5076e-01 + 3.5785e-01i    \;&\;  9.5020e-16   \;&\;  1.1102e-16  \;&\;  4.0030e-16 \\
				
			    3.5076e-01 - 3.5785e-01i \;&\;  9.5020e-16   \;&\;   1.1102e-16  \;&\; 4.0030e-16 \\
			    
				-1.2170e-01 + 6.2287e-01i    \;&\;  6.7589e-16   \;&\;  7.8945e-16 \;&\;   2.2248e-16\\
				
					-1.2170e-01 - 6.2287e-01i    \;&\;  6.7589e-16  \;&\; 7.8945e-16    \;&\;  2.2248e-16\\
					
					-4.5691e+02                   \;&\;  2.9559e-12   \;&\; 2.7285e-12  \;&\; 5.6843e-14 \\
					
					Inf   &    NaN    &   NaN   &  (Inf)    \\
					Inf   &    NaN    &   NaN    &  (Inf)    \\
					Inf   &    NaN    &   NaN   &  (Inf)  \\
					Inf   &    NaN    &   NaN   &  (Inf)  \\
				\end{tabular} 	\end{center}	
			
				\caption{The correct generalized $\la_i$ and the corresponding accuracies $\delta^{k}_i$ for the three different calculations} 
				
		\end{table} 
      	\vspace{-0.4cm}
	

We notice that for the largest finite eigenvalue of the order of $10^2$ the $QZ$ algorithm applied to $S_P(\la)$ gets 14 digits of relative accuracy but, when deflating the four uncontrollable eigenvalues at $\infty$, our method recovers a relative accuracy of 16 digits.
	\end{example}

	\begin{example}	{Example 2} The second example is the rational matrix $R(\la)$ in \eqref{Ex1} with $c=1$.
		$$ R(\la) = \left[\begin{array}{cc} e_5(\lambda) \;&\; 0 \\ 1/\la \;&\; e_1(\lambda) \end{array}\right]=P_0+ P_1\lambda + \cdots + P_5 \lambda^5 +  \left[\begin{array}{cc} 0 \;&\; 0 \\ 1/\la \;&\; 0 \end{array}\right],$$
		by using the notation of the example above. In this case, $e_5(\lambda)$ has the row vector $[  4.7865e-02 \;\;
		1.4279e-04 \;\;
		2.4361e-03 \;\;
		-1.5336e-02 \;\;
		-9.9155e-01 \;\;
		1.1948e-01 ]$ as coefficients, and $e_1(\lambda)$ has the row vector $[6.5250e-03 \;\;
		9.9997e-01 ]$. We consider the $12 \times 12$ linear system matrix
		$$ S_R(\lambda) =  \left[\begin{array}{ccccc|c} \lambda I_2 -A & & & & & -B \\ & I_2 \;&\; -\lambda I_2 & & & P_1 \\  & & I_2 \;&\; -\lambda I_2 & & P_2 \\
		& & & I_2 \;&\; -\lambda I_2 & P_3 \\ & & & &  I_2 \;&\; P_4 + \lambda P_5 \\ \hline
		C & -\lambda I_2 & & & & P_0
		\end{array}\right],
		$$
		where 
		$$ A= \left[\begin{array}{cc} 0 \;&\; 0 \\ 1 \;&\; 0 \end{array}\right], \quad B= \left[\begin{array}{cc} 0 \;&\; 0 \\  1 \;&\; 0 \end{array}\right]
		\quad C= \left[\begin{array}{cc} 0 \;&\; 0 \\ 0 \;&\; 1 \end{array}\right]$$
		is a non-minimal realization of the strictly proper rational function $1/\la$. In fact, the matrix $A$ in the realization triple $(A,B,C)$ has two eigenvalues at $\la=0,$ of which one is uncontrollable since $1/\la$ only has a pole at $0$ of order $1.$
		This is an artificial example since we could have realized the strictly proper part by using a minimal triple $(A,B,C)$ by removing the uncontrollable 
		eigenvalue, but this is precisely what our reduction procedure does simultaneously for finite and infinite uncontrollable eigenvalues. 
		The quantities given in the following table are defined as in the previous example, except that we added two roots at $0$ corresponding to the exact eigenvalues.
		
		\vspace{0.4cm}
		\begin{table}[ht] 
				\begin{center}	\begin{tabular}{c|ccc} 
					$\la_i$ \;&\;$\delta^{(2)}_i$ \;&\;  $\delta^{(3)}_i$ \;&\; $\delta^{(4)}_i$ \\ \hline
					0                                        \;&\;  0                 \;&\; 8.1752e-09  \;&\;  (4.5874e-16) \\ 
					0                                        \;&\; 3.6752e-18   \;&\; 8.1752e-09  \;&\; 5.3729e-16 \\
				1.2028e-01                             \;&\;  1.8041e-16   \;&\; 9.7145e-17   \;&\;  9.7145e-17 \\
				2.1135e+00                            \;&\;  1.7764e-15   \;&\; 2.6645e-15 \;&\; 1.3323e-15 \\
	       	-2.1404e+00                              \;&\; 1.7764e-15   \;&\; 2.2204e-15  \;&\; 8.8818e-16\\
			-4.8180e-02 + 2.1412e+00i        \;&\; 2.3216e-15      \;&\; 1.7990e-15   \;&\; 4.0614e-15 \\
			-4.8180e-02 - 2.1412e+00i         \;&\; 2.3216e-15    \;&\; 1.7990e-15  \;&\; 4.0614e-15\\
			-1.5325e+02                               \;&\; 2.5580e-13   \;&\; 1.5321e-07   \;&\; 5.6843e-14\\
					Inf   &    NaN    &   Inf   &  (Inf)    \\
					Inf   &    NaN    &   Inf   &  (Inf)  \\
					Inf   &    NaN    &   NaN   &  (NaN)  \\
					Inf   &    NaN    &   NaN   &  (NaN)  \\
				\end{tabular}	\end{center}
	\caption{The correct generalized $\la_i$ and the corresponding accuracies $\delta^{k}_i$ for the three different calculations}
		\end{table}
		\vspace{-0.4cm}
	
		In this example the $QZ$ algorithm applied to $S_R(\la)$ recovers well all generalized eigenvalues. When applying the $QZ$ algorithm to an orthogonally equivalent pencil $QS_R(\la)Z$, the Jordan block at 0 gets perturbed to two roots of the order of the square root of the machine precision, which can be expected. But when deflating the uncontrollable eigenvalue at 0, this Jordan block is reduced to a single eigenvalue and part of the accuracy gets restored.
	\end{example}
\vspace{-0.1cm}
These two examples show that deflating uncontrollable eigenvalues may improve the sensitivity of the remaining eigenvalues which
may improve the accuracy of their computation.
  
\section{Conclusion} \label{Conclusion}

In this paper we looked at quadruple realizations $\{A(\la),B(\la),C(\la),D(\la)\}$ for a given rational transfer function 
$R(\la)=C(\la)A(\la)^{-1}B(\la)+D(\la)$, where  the matrices $A(\la), B(\la), C(\la)$ and $D(\la)$ are pencils, and where 
$A(\la)$ is assumed to be regular. We showed that under certain minimality assumptions on this quadruple, the poles, zeros 
and left and right null space structure of the rational matrix $R(\la)$ can be recovered from the generalized eigenstructure of 
two block pencils constructed from the quadruple. We also showed how to obtain such a minimal quadruple from a non-minimal one, by applying a reduction procedure that is based on the staircase algorithm. These results extend those previously obtained for generalized state space systems and polynomial matrices. 

\begin{acknowledgement} We would like to thank the anonymous reviewer whose helpful comments and suggestions have greatly improved this manuscript. The first author was supported by ``Ministerio de Econom\'ia, Industria y Competitividad (MINECO)'' of Spain and ``Fondo Europeo de Desarrollo Regional (FEDER)'' of EU through grants MTM2015-65798-P and MTM2017-90682-REDT. The second author was funded by the “contrato predoctoral” BES-2016-076744 of MINECO. This work was developed while the third author held a ``Chair of Excellence UC3M - Banco de Santander'' at Universidad Carlos III de Madrid in the academic year 2017-2018. 
\end{acknowledgement}

\bibliographystyle{plain}

\end{document}